\documentclass[12pt]{article}

\usepackage{amsrefs}
\usepackage{amsmath,amssymb,amsthm}
\usepackage{enumerate,pxfonts}
\usepackage{tikz}
\usetikzlibrary{patterns}
\usepackage{tkz-euclide}
\usetkzobj{all}
\usepackage[all,cmtip]{xy}
\usepackage[applemac]{inputenc}
\usepackage[colorlinks=true, linkcolor=black, citecolor=black]{hyperref}

\def \1{\mathds{1}}
\def \a{\mathfrak a}
\def \al{\alpha}
\def \b{\mathfrak b}
\def \be{\beta}
\def \bs{\backslash}
\def \C{{\mathbb C}}

\def \CB{\mathcal{B}}
\def \CC{\mathcal{C}}

\def \CE{\mathcal{E}}

\def \CS{\mathcal{S}}

\def \cusp{\mathrm{cusp}}
\def \D{{\mathbb D}}
\def \Eig{\operatorname{Eig}}
\def \End{\operatorname{End}}
\def \eps{\varepsilon}
\def \even{\mathrm{even}}

\def \Frob{\mathrm{Frob}}
\def \Ga{\Gamma}
\def \GL{\operatorname{GL}}
\def \ga{\gamma}
\def \H{{\mathbb H}}

\def \Id{{\rm Id}}
\def \Im{\operatorname{Im}}
\def \la{\lambda}
\def \mqed{\tag*\qedhere}
\def \N{{\mathbb N}}
\def \nequiv{\equiv\hspace{-11pt}/\hspace{5pt}}
\def \ol{\overline}
\def \om{\omega}
\def \Om{\Omega}
\def \op{\mathrm{op}}

\def \PSL{\operatorname{PSL}}

\def \R{{\mathbb R}}
\def \Re{\operatorname{Re}}

\def \sm{\smallsetminus}
\def \SO{\operatorname{SO}}

\def \ul{\underline}

\def \what{\widehat}
\def \Z{{\mathbb Z}}
\def \({\left(}
\def \){\right)}
\def \[{\Big\{}
\def \]{\Big\}}

\newcommand{\e}
[1]{\emph{#1}\index{#1}}

\newcommand{\mat}
[4]{\(\begin{matrix}#1 & #2 \\ #3 & #4\end{matrix}\)}

\newcommand{\norm}
[1]{\left\|#1\right\|}
 
\newcommand{\smat}
[4]{\(\begin{smallmatrix}#1 & #2 \\ #3 & #4\end{smallmatrix}\)}

\renewcommand{\sp}
[1]{\left\langle #1\right\rangle}

\newtheorem{theorem}{Theorem}[section]

\newtheorem{lemma}[theorem]{Lemma}

\newtheorem{proposition}[theorem]{Proposition}

\theoremstyle{definition}
\newtheorem{definition}[theorem]{Definition}

\begin{document}

\pagestyle{myheadings} \markright{Eisenstein series with non-unitary twists}

\title{Eisenstein series with non-unitary twists}
\author{Anton Deitmar \& Frank Monheim}
\date{}
\maketitle

{\bf Abstract:} It is shown that for a non-unitary twist of a Fuchsian group, which is unitary at the cusps, Eisenstein series converge in some half-plane. It is shown that invariant integral operators provide a spectral decomposition of the space of cusp forms and that Eisenstein series admit a meromorphic continuation.



$$ $$

\tableofcontents

\newpage
\section*{Introduction}

Werner M\"uller established in \cite{mueller} a trace formula with
non-unitary twists for compact locally symmetric spaces
$\Gamma\backslash G/K$. In the paper \cite{DM} the authors extend this to
compact quotients $\Gamma\backslash G$, for a Lie group or a totally
disconnected group $G$.
This paper is the first step towards a corresponding formula in the case of non-compact quotient $\Ga\bs G$.
In this paper we specialize to the case $G=\PSL_2(\R)$. We fix a non cocompact lattice $\Ga$ in $G$. For convenience we can switch to a subgroup of finite index and  thus assume $\Ga$ to be torsion-free.

First we set up the notion of a canonical Hilbert space. It turns out that the Hilbert structure is not canonical, though.
Using geometric estimates  on the word representations of group elements, it is shown that Eisenstein series actually do converge. For this, however, it is needed that the representation $\chi:\Ga\to\GL(V)$ be bounded in terms of the, say, Frobenius norm of the elements of the Fuchsian group $\Ga$.
Recall that a Fuchsian group with cusps, which is torsion-free, is a free group in finitely many generators \cite{Katok}. 
So in order to give a representation $\chi$ as above, one only needs to give one arbitrary matrix $\chi(\tau)$ for each generator $\tau$.
If the group has at least two cusps, the canonical set of generators contains a parabolic element, but if one has a parabolic generator, say $\tau=\smat 11\ 1$ and maps it to a semisimple matrix, then the norms $\norm{\chi(\tau^k))}$ will tend to infinity exponentially in $k$, while the Frobenius norm equals $\norm{\tau^k}_{\Frob}=\sqrt{2+k^2}$.
So in that case an estimate of the desired kind is impossible and consequentially, the corresponding Eisenstein series won't converge.
The solution is to insist that the representation be unitary on parabolic elements, which is a restriction which is used in the text.
Under this condition indeed we show convergence and analyticity of the Eisenstein series.
As a by-product we get that the space of cusp forms is stable under all invariant operators.

We thank Ksenia Fedosova for pointing out a mistake in an earlier version of this paper.

\section{Notation}

The group $G=\PSL_2(\R)$ acts transitively on the upper half plane $\H$ in $\C$ by linear fractionals.
Let $\Ga\subset G$ be a lattice, i.e., a discrete subgroup of finite covolume.
We will throughout assume that $\Ga\bs \H$ is non-compact, i.e., the group $\Ga$ has \e{cusps}.
Then $\Ga$ acts on $\H$ properly discontinuously.
The action extends to the boundary $\partial\H=\R\cup\{\infty\}$, and for each $w\in\ol\H=\H\cup\partial\H$ we write $\Ga_w$ for its stabilizer, i.e.,
$$
\Ga_w=\big\{ \ga\in\Ga: \ga w=w\big\}.
$$
A connected open set $F\subset\H$, whose boundary is a null-set, is called a \e{fundamental domain} for $G$ if there exists a set of representatives $R$ for $\Ga\bs\H$ such that
$$
F\subset R\subset \ol F.
$$
An example is the \e{Dirichlet domain} $D(z_0)$ to a point $z_0\in\H$ with trivial stabilizer $\Ga_{z_0}=\{1\}$:
$$
D(z_0)=\[ z\in\H: d(z,z_0)<d(z,\ga z_0)\text{ for every }1\ne \ga\in\Ga\].
$$
We will fix a domain $F$ of this type. It has finitely many geodesic sides and finitely many cusps.
For each cusp $\a\in\what\R=\R\cup\infty$ there exists an element $\sigma_\a\in G$ such that
\begin{itemize}
\item $\sigma_\a\infty=\a$,
\item $\sigma_\a^{-1}\Ga_\a\sigma_\a=\pm\mat 1\Z\ 1.$
\end{itemize}
In order to measure the location of a point $z$ in the Riemann surface $\Ga\bs \H$ with respect to the compact core of $\Ga\bs \H$ and the cuspidal ends, we introduce the \e{invariant height}
$$
y_\Ga(z)=\max_\a\max_{\ga\in\Ga}\(\Im(\sigma_\a^{-1}\ga z)\).
$$ 
We say that $z\in\H$ approaches the cusp $\a$ if $\Im(\sigma_\a^{-1}\ga z)\to\infty$.
For $Y>0$ let $P(Y)$ be the set of all $z=x+iy\in\H$ with $0<x< 1$ and $y\ge Y$.
For  large enough $Y$ the scaling matrix $\sigma_\a$ maps this strip injectively into $F$. Let $F_\a(Y)=\sigma_\a P(Y)$ denote the image. For large enough $Y$ we set
$$
F(Y)=F\sm\(\bigcup_\a F_\a(Y)\).
$$
Then $F(Y)$ is relatively compact in $\H$ and we have divided the fundamental domain $F$ into the central part $F(Y)$ and finitely many cuspidal zones
$$
F=F(Y)\cup\bigcup_\a F_\a(Y).
$$
  
  \begin{equation*}
\begin{tikzpicture}[scale=1]
  \tkzDefPoint(0,0){0}
  \tkzDefPoint(1,0){A}
  \tkzDefPoint(2,0){B}
  \tkzDefPoint(3,0){C}
  \tkzDefPoint(4,0){D}
  \tkzDrawLine(0,D)
  \tkzDrawArc[color = black](B,C)(A) 
  \tkzDrawArc[rotate, color = black](0,A)(110)
  \tkzDrawArc[rotate, color = black](D,C)(-110)
  \tkzDefPoint(110:1){E0}
  \tkzDefPoint(70:1){E}
  \tkzDefPoint[shift = {(2,0)}](110:1){F} 
  \tkzCircumCenter(A,E,F)\tkzGetPoint{I} 
  \tkzDefPoint[shift = {(2,0)}](70:1){G}
  \tkzDefPoint[shift = {(4,0)}](110:1){H}
  \tkzDefPoint[shift = {(4,0)}](70:1){H0}
  \tkzCircumCenter(C,G,H)\tkzGetPoint{J}
  
  \tkzDefPoint(110:1){K}
  \tkzDefPoint[shift = {(4,0)}](70:1){L}
  \tkzDrawArc(I,F)(E)
  \tkzDrawArc[style = dashed](I,E)(F)
  \tkzDrawArc(J,H)(G)
  \tkzDrawArc[style = dashed](J,G)(H)
  
  \tkzDefShiftPoint[E0](0,4){M}
  \tkzDefShiftPoint[H0](0,4){N}
  
  \tkzDrawLine[add = 0 and .2](E0,M)
  \tkzDrawLine[add = 0 and .2](H0,N)
  
  \tkzDefShiftPoint[M](-0.5,0){M0}
  \tkzDefShiftPoint[N](0.5,0){N0}
  \tkzDrawLine[add = 0 and 0, color = gray, line width = 0.4 pt](M,N)
  \tkzDrawLine[style = dashed, color = gray, add = 0.2 and 0, line width = 0.4 pt](M0,M)
  \tkzDrawLine[style = dashed, color = gray, add = 0.2 and 0, line width = 0.4 pt](N0,N)
  
  \tkzLabelPoint[below](A){$\mathfrak{a}$}
  \tkzLabelPoint[below](C){$\mathfrak{b}$}
  \tkzLabelPoint[below](2,3){$F(Y)$}
  \tkzLabelPoint[above](I){$F_\mathfrak{a}(Y)$}
  \tkzLabelPoint[above](J){$F_\mathfrak{b}(Y)$}
  \tkzLabelPoint[above](2,5){$F_\infty(Y)$}
\end{tikzpicture}
\end{equation*}

\begin{definition}
Let 
$$
\chi:\Ga\to\GL(V)
$$
be a finite-dimensional representation of $\Ga$ on a unitary space $V$.
We say that $\chi$ is \e{unitary at the cusps}, if for every cusp $\a$ and every $\ga\in\Ga_\a$ the endomorphism $\chi(\ga)$ of $V$ is unitary.

The orthogonal projection $V\to V^{\chi(\Ga_\a)}$ onto the space of $\Ga_\a$ fixed points will be denoted by $P_\a$.
In the case that $P_\a=0$ for every cusp we say that $\chi$ is \e{non-singular}.
\end{definition}

We choose a fundamental domain $F\subset\H$
and define
$$
L^2(F,\chi)=\left\{ f:\H\to V: \begin{array}{c}f\text{ is measurable}\\ f(\ga z)=\chi(\ga)f(z),\ {\ga\in\Ga,\ z\in\H}\\ \int_F\norm{f(z)}_V^2\,dz<\infty\end{array}\right\}
$$
modulo nullfunctions.
This Hilbert space depends on the choice of $F$, but only in a mild way.

\begin{definition}
Two fundamental domains $F_1$ and $F_2$ are said to be \e{equivalent} if, up to a set of measure zero, $F_2$ can be covered by finitely many $\Ga$-translates of $F_1$ and vice-versa.
\end{definition}

\begin{proposition}
If two fundamental domains $F_1$ and $F_2$ are equivalent, then $L^2(F_1,\chi)$ and $L^2(F_2,\chi)$ coincide as sets and the identity map 
$$
L^2(F_1,\chi)\to L^2(F_2,\chi)
$$
is a topological isomorphism.
\end{proposition}

\begin{proof}
Let $\ga_1,\dots,\ga_l\in\Ga$ such that $F_2\subset \(N\cup\bigcup_{j=1}^l\ga_jF_1\)$, where $N$ is a null-set. Then
$$
\int_{F_2}\norm{f(z)}^2\,dz\le\sum_{j=1}^l \norm{\chi(\ga_j)}^2\int_{F_1}\norm{f(z)}^2\,dz,
$$
where $\norm{\chi(\ga_j)}$ is the operator norm. By symmetry the claim follows.
\end{proof}

\begin{definition}
A fundamental domain $F$ is called \e{geometrically finite}, if it is bounded by finitely many geodesics.
A Dirichlet domain is an example.
\end{definition}

\begin{proposition}
Any two geometrically finite fundamental domains are equivalent.
\end{proposition}

\begin{proof}
We consider the Borel-Serre compactification.
Let 
$$
\H_\Ga = \H\cup\bigcup_\a\(\partial\H\sm\{\a\}\).
$$
We write $B_\a$ for $\partial\H\sm\{\a\}$.
For the points of $B_\a$ we will write $b_\a(x)$ where $x\in \partial\H\sm\{\a\}$ to distinguish the points of different copies of $\partial\H=\R\cup\{\infty\}$.
We install a topology on $\H_\Ga$ as follows
\begin{itemize}
\item any $z\in\H$ has the usual neighborhood base in $\H$,
\item for $x\in \partial\H\sm\{\a\}$ we define a neighborhood base by switching to the disk model $\D$ of the hyperbolic space. On the disk, $d_e(x,y)$ shall denote the euclidean distance of $x$ and $y$. Then a neighborhood basis will be given by the sets $U_{I,\eps}=I\cup U_\eps$, where $I\subset B_\a$ is open and 
$U_\eps$ is the set of all $z\in \D$ which satisfy $d_e(z,\a)<\eps$ and $z$ lies on a geodesic joining $\a$ and a point in $I$. 
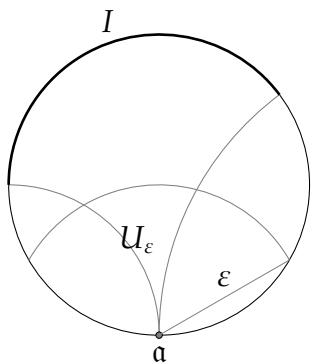
\begin{equation*} 
  \begin{tikzpicture}[scale=2]
    \tkzDefPoint(0,0){0} 
    \tkzDefPoint(1,0){A}
    \tkzDrawCircle[line width = 0.4 pt](0,A)
    \tkzDefPoint(2,-1){B}
    \tkzDefPoint(-1,0){E}
    \tkzDefPoint(-1,-1){F}
    \tkzDefPoint(0,-1){G}
    \tkzDefCircle[orthogonal from=B](0,A)\tkzGetFirstPoint{C}\tkzGetSecondPoint{D}
    \tkzDrawArc(B,C)(D)
    \tkzDrawArc(F,G)(E)
    \tkzInterCC[R](0, 1cm)(G, 1 cm) \tkzGetPoints{I}{J}
    \tkzDrawArc(G,I)(J)
    \tkzDefPoint(110:1){M}
    \tkzLabelPoint[above](M){$I$}
    \tkzDefShiftPoint[0](.05,-0.2){N}
    \tkzLabelPoint[below left](N){$U_\varepsilon$}
    \tkzLabelPoint[below](G){$\mathfrak{a}$}
   
    \tkzDrawPoint(G)
    \tkzDrawArc[line width = 1 pt, color = black](0,C)(E)
    \tkzDrawLine[add = 0 and 0, color = gray, line width = 0.2 pt](G,I)
    \tkzLabelLine[above](G,I){$\varepsilon$}
  \end{tikzpicture} 
\end{equation*}
\end{itemize}
The group $\Ga$ acts on $\H_\Ga$ in the following way. On $\H$ it is the usual action by hyperbolic isometries. For a point $b_\a(x)$ we set $\ga\cdot b_\a(x)=b_{\ga\a}(\ga x)$. It is easy to see that this $\Ga$-action is properly discontinuous.
A fundamental domain for this $\Ga$-action is given as follows: Let $F$ be a geometrically finite fundamental domain for the $\Ga$-action on $\H$.
It has finitely many inequivalent cusps.
For each cusp $\a$ of $F$ we choose $a_\a,b_\a\in B_\a\cong\R$ such that the unique geodesics $\al_\a,\beta_\a$ which join $\a$ to $a_\a$ and $b_\a$ respectively contain the two faces of $F$ which meet at $\a$.
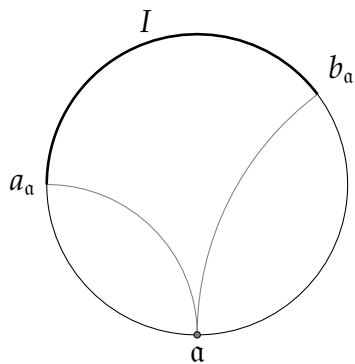
\begin{equation*}
  \begin{tikzpicture}[scale=2]
    \tkzDefPoint(0,0){0} 
    \tkzDefPoint(1,0){A}
    \tkzDefPoint(0,1){H}
    \tkzDrawCircle[line width = 0.4 pt](0,A)
    \tkzDefPoint(2,-1){B}
    \tkzDefPoint(-1,0){E}
    \tkzDefPoint(-1,-1){F}
    \tkzDefPoint(0,-1){G}
    \tkzDefCircle[orthogonal from=B](0,A)\tkzGetFirstPoint{C}\tkzGetSecondPoint{D}
    \tkzDrawArc(B,C)(D)
    \tkzDrawArc(F,G)(E)
    \tkzInterCC[R](G, 1cm)(B, 2 cm) \tkzGetPoints{I}{J}
    \tkzDefPoint(110:1){M}
    \tkzLabelPoint[above](M){$I$}
    \tkzDefShiftPoint[0](.05,-0.2){N}
    \tkzLabelPoint[below](G){$\mathfrak{a}$}
    \tkzLabelPoint[left](E){$a_\mathfrak{a}$}
    \tkzLabelPoint[above right](C){$b_\mathfrak{a}$}
   
    \tkzDrawPoint(G)
    \tkzDrawArc[line width = 1 pt, color = black](0,C)(E)
   \end{tikzpicture} 
\end{equation*}
Then $F_\Ga=F\cup\bigcup (a_\a,b_\a)$ can be shown to be a fundamental domain of the $\Ga$-action on $\H_\Ga$.
The closure of $F_\Ga$ in $\H_\Ga$ is given by
$$
\ol{F_\Ga}^{\H_\Ga}=\ol F^\H\cup\bigcup_\a [a_\a,b_\a].
$$
As $F$ has only finitely many faces, it follows that $\ol{F_\Ga}^{\H_\Ga}$ is compact.
Consider two geometrically finite fundamental domains $F,F'\subset\H$.
Since $F_\Ga$ and $F'_\Ga$ are both relative compact and the $\Ga$-action is properly discontinuous, one can cover $F'_\Ga$ by finitely many $\Ga$-translates of $\ol{F_\Ga}$ and vice-versa. Taking the intersection with $\H$, the claim follows.
\end{proof}

If $\Ga$ is torsion-free and $\chi$ is unitary at the cusps, we can yet introduce an $L^2$-space in a different manner, such that again we get a topological isomorphism to the former spaces.
We let $\Ga$ act on $\H\times V$ diagonally. The quotient space
$$
\H\times_\Ga V=\Ga\bs (\H\times V)
$$
yields a flat vector bundle over $\Ga\bs\H$ with fibre $V$.
We choose a hermitean fibre metric $\sp{.,.}_s$ in such a way that it coincides with the given inner product near the cusps.
This means that for $z\in\H$ whose invariant height $y_\ga(z)$ is larger than some constant, say $y_\ga(z)>c$, the inner product on the fibre above $z$ coincides with the inner product on $V$. It is possible to find such a fibre metric, as the representation $\chi$ is unitary at the cusps.
Now define the Hilbert space $L^2(\Ga\bs\H,\chi)_s$ as the set of all measurable sections $f:\Ga\bs\H\to\H\times_\Ga V$ such that
$$
\int_{\Ga\bs \H}\norm{f(z)}_s^2\,dz<\infty
$$
modulo nullfunctions.
A section $f$ can be considered a map $f:\H\to V$ with $f(\ga z)=\chi(z)f(z)$ and so we find

\begin{proposition}
Let $\Ga$ be torsion-free and $\chi$ unitary at the cusps. Then for any geometrically finite fundamental domain $F$ the sets $L^2(\Ga\bs\H,\chi)_s$ and $L^2(F,\chi)$ coincide and the identity map yields a topological isomorphism of Hilbert spaces
$$
L^2(F,\chi)\to L^2(\Ga\bs \H,\chi)_s.
$$ 
\end{proposition}

\begin{proof}
We pull back the smooth fibre metric $\sp{.,.}_s$ to the trivial bundle $\H\times V$ and denote this metric by $\sp{\sp{.,.}}$.
Since any two norms on a finite-dimensional space are equivalent and since $\sp{.,.}$ and $\sp{\sp{.,.}}$ coincide on cuspidal areas $F_\a(Y)$ for large enough $Y$, there exist constants $m,M>0$ such that
$$
m\sp{\sp{(z,v),(z,v)}}\le \sp{v,v}\le M\sp{\sp{(z,v),(z,v)}}
$$
holds for all $(z,v)\in F\times V$.
The claim follows.
\end{proof}

For this smooth fibre metric there exists corresponding Laplacian $\Delta_s$ which has the same principal symbol as the hyperbolic Laplacian
$$
\Delta=-y^2\(\frac{\partial^2}{\partial x^2}+\frac{\partial^2}{\partial y^2}\).
$$
The sign is chosen to make the operator positive semidefinite.
Let $(.,.)_s$ denote the inner product of $L^2(\Ga\bs \H,\chi_s)$.
We introduce the Sobolev spaces
\begin{align*}
H^1(\Ga\bs \H,\chi)_s&=\[ f\in L^2(\Ga\bs \H,\chi)_s: (f,f)_s+(\Delta_s f,f)_s<\infty\],\\
H^2(\Ga\bs \H,\chi)_s&=\[ f\in L^2(\Ga\bs \H,\chi)_s: (f,f)_s+(\Delta_s f,\Delta_s f)_s<\infty\],
\end{align*}
where $\Delta_sf$ is understood in the distributional sense.
For later use note that the function $f(z)=\Im(z)^s$ satisfies $\Delta f=s(1-s)f$.

\section{Estimating representations}
We assume $\Ga$ to be a torsion-free lattice with cusps and $\chi$ a finite-dimensional representation of $\Ga$ which is unitary at the cusps.

The classical Eisenstein series for a cusp $\a$ is defined by
$$
E_\a(z,s)=\sum_{\ga\in\Ga_\a\bs\Ga} \Im(\sigma_\a^{-1}\ga z)^s,
$$
the series being convergent for $z\in\H$, $s\in\C$, $\Re(s)>1$.
For the $\chi$-twist we define the Eisenstein series by
$$
E_\a(z,s,\chi)=\sum_{\ga\in\Ga_\a\bs\Ga}\Im(\sigma_\a^{-1}\ga z)^s\,\chi(\ga^{-1})\,P_\a.
$$
We will have to show convergence.

{\bf Notation.} In the following, we shall use the ``big O'' and the ``$\ll$'' notation: for a set $X$ and two functions $f:X\to\C$ and $h:X\to (0,\infty)$ we write
$$
f=O(h),\quad\text{or}\quad f\ll h,
$$
if there exists a constant $C>0$, such that 
$$
|f(x)|\le C\, h(x)
$$
holds for every $x\in X$.
Any choice of $C$ will be referred to as the \emph{implied constant}.

\begin{proposition}
There exists $\al>0$ such that the operator norm satisfies
$$
\norm{\chi(\ga)}=O(\Im(\ga z)^\al),
$$
where the implied constant depends continuously on $z\in\H$.
This means that there exists a continuous function $C:\H\to (0,\infty)$ such that 
$$
\norm{\chi(\ga)}\le C(z)\Im(\ga z)^\al
$$
holds for all $\ga\in\Ga$ and all $z\in\H$.

In particular, the Eisenstein series $E_\a(z,s,\chi)$ converges locally uniformly absolutely for $\Re(s)>1+\al$. 
\end{proposition}

The proof of this proposition occupy the rest of the section.

First observe that, as $\chi$ is unitary at the cusps, we have
$$
\norm{\chi(\ga_\a\ga)}=\norm{\chi(\ga)}
$$ 
for $\ga\in\Ga$ and $\ga_\a\in\Ga_\a$.

\begin{definition}
(Normal presentation)
We fix a fundamental domain $F$ of the form $F=D(z_0)$. Let $\ga_jF$ with $j=1,\dots,l$ be the neighboring domains. Then the set $\{\ga_1,\dots,\ga_l\}$ is a symmetric generating set of the group $\Ga$.

A presentation $\ga=\eta_1\eta_2\cdots\eta_r$ with $\eta_i\in \{\ga_1,\dots,\ga_l\}$ is called \e{normal}, if the $\eta_i$ are chosen in the following manner:
We fix a second point $z_1\in F$ different from $z_0$ and chosen such that for any two $\tau_0,\tau_1\in\Ga$ the geodesic joining $\tau_0 z_0$ to $\tau_1 z_1$ is different from any geodesic being a boundary line of any fundamental domain $\sigma F$, $\sigma\in \Ga$.
We assume $\eta_1,\dots\eta_{j-1}$ to be already found.
Now join $\eta_1\cdots\eta_{j-1}z_0$ with $\ga z_1$ by a geodesic. Following this geodesic from $\eta_1\cdots\eta_{j-1}z_0$ in the direction of $\ga z_1$, after leaving the fundamental domain $\eta_1\cdots\eta_{j-1}F$ the geodesic enters a fundamental domain of the form $\eta_1\cdots\eta_{j-1}\ga_kF$ with $1\le k\le l$.
We set $\eta_j=\ga_k$.
As the geodesic distance $d(\eta_1\cdots\eta_j z_0,\ga z_1)$ decreases, this procedure will stop and give a presentation of $\ga$.
\end{definition}

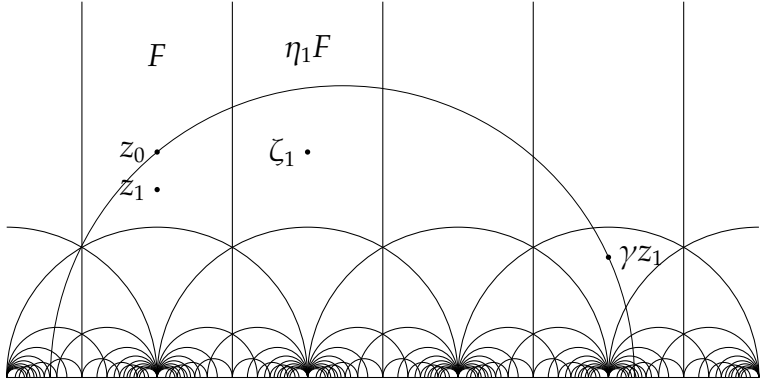
\begin{equation*}
  \begin{tikzpicture}[scale=2]
    \draw (0,0) -- (5,0); \foreach \x in {0,...,3} { \draw (\x,0) arc
      (180 : 0 : 1);} \draw (1,0) arc (0: 90 :1); \draw (4,0) arc (180
    : 90 : 1); \foreach \y in {0,...,4}{ \foreach \x in {1,...,8}{
        \draw (\y,0) arc (180 : 0 : {1.0/(\x *2 + 1)}); \draw ({\y
          +1},0) arc (0 : 180 : {1.0/(\x *2 + 1)}); }} \foreach \y in
    {0,...,4}{ \foreach \x in {1,...,4}{ \draw ({1.0/(\x + 1.0) +
          \y},0) arc (0 : 180 : {1.0/(2 * \x + 6)});}} \foreach \y in
    {0,...,4}{ \foreach \x in {1,...,4}{ \draw ({1.0 - (1.0/(\x +
          1.0)) + \y},0) arc (180 : 0 : {1.0/(2 * \x + 6)});}}
    \foreach \x in {0,...,4}{ \draw ({\x + 0.5}, 0) -- +(0,2.5);}
    \coordinate[label=left: $z_{0}$] (z_0) at (1,1.5); \fill (1,1.5)
    circle (0.5pt); \coordinate[label=left: $z_{1}$] (z_1) at
    (1,1.25); \fill (1,1.25) circle (0.5pt); \coordinate[label=right:
    $\gamma z_{1}$] (gz) at (4,0.8); \fill (gz) circle (0.5pt);
    \draw[very thin] ({-sqrt(((9 + 16.0/25 - 2.25)/6)^2 + (2.25)) +1 +
      (9 + 16.0/25 - 2.25)/6},0) arc (180 : 0 : {sqrt(((9 + 16.0/25 -
      2.25)/6)^2 + (2.25))}); \coordinate[label=above: $F$] (F) at
    (1,2); \coordinate[label=above: $\eta_1F$] (etaF) at (2,2); \fill
    (2,1.5) circle (0.5pt); \coordinate[label=left: $\zeta_1$]
    (zeta_1) at (2,1.5);

  \end{tikzpicture}
\end{equation*}

We write $\zeta_j=\eta_1\cdots\eta_jz_0$.

\begin{lemma}\label{lem2.3}
There exists a positive constant $c_1$ independent of $\ga$ and $z_1$ such that for the hyperbolic distance we have
$$
d(\zeta_j,\ga z_1)\le d(\zeta_{j-1},\ga z_1)-c_1,
$$
if $\eta_j$ is not parabolic.
\end{lemma}

\begin{proof}
We recall that the fundamental domain $F$ is of the form $F=\{ z\in\H: d(z,z_0)<d(z,\ga z_0)\text{ for all }\ga\in\Ga\}$.
Let $g(\zeta_{j-1},\ga z_1)$ denote the geodesic joining $\zeta_{j-1}$ and $\ga z_1$.
Let $z\in g(\zeta_{j-1},\ga z_1)$ be the unique point between $\zeta_{j-1}$ and $\ga z_1$ which lies in the boundary of $\eta_1\cdots\eta_{j-1} F$. 

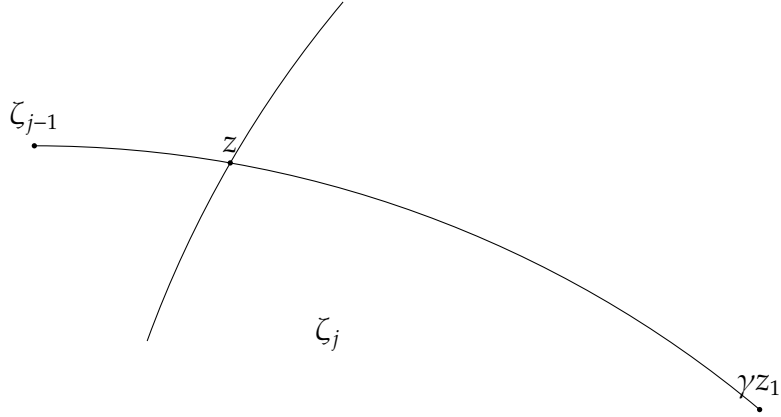
\begin{center}
\begin{tikzpicture}
    \draw (0,15) arc (90: 50 : 15); \coordinate[label=above:
    $\zeta_{j-1}$] (z_{i-1}) at (0,15); \fill (z_{i-1}) circle (1pt);
    \coordinate[label = above: $\gamma z_1$] (z) at (50:15); \fill (z)
    circle (1pt);  
    \coordinate[label = right :$\zeta_{j}$] (zeta_i) at (74:13);
    \coordinate[label = above : $z$] (eta) at (80:15); \fill
    (80:15) circle (1pt);       
    \draw (80:15) arc (150 : 160 : 15); \draw (80:15) arc (150 : 140 :
    15);
\end{tikzpicture}
\end{center}

Since the fundamental domain $\eta_1\cdots\eta_jF$ is the dirichlet domain with center $\zeta_j$, all of the geodesic line, which is to the right of $z$, is nearer to $\zeta_j$ than to $\zeta_{j-1}$. For $\ga z_1$ this means 
$$
0\le d(\zeta_{j-1},\ga z_1)-d(\zeta_j,\ga z_1),
$$
where we would have equality only if $\ga z_1=z$, which is impossible.
So we get
$$
d(\zeta_{j-1},\ga z_1)-d(\zeta_j,\ga z_1)\ge c_1
$$
with some $c_1>0$. We have  make clear that we can choose $c_1$ independent of $\ga$.
This becomes clear by
$$
d(\zeta_{j-1},\ga z_1)-d(\zeta_j,\ga z_1)=
d(z_0,\ga' z_1)-d(\eta_j z_0,\ga' z_1)
$$
with $\ga'=\eta_1\cdots\eta_{j-1}\ga$, which reduces the claim to the case $j=1$.
In this case, as $\eta_j$ is not parabolic, $z$ can only vary in the compact part of the boundary of $F$ where $\ol F$ meets $\eta_j\ol F$. This implies independence of $\ga$.
\end{proof}

\begin{lemma}\label{lem2.4}
There exists a constant $c_2>0$, independent of $\ga$ and $z_1$ such that
$$
d(\zeta_{j+1},\ga z_1)\le d(\zeta_{j-1},\ga z_1)-c_2
$$
as long as $\eta_j,\eta_{j+1}$ are not parabolic elements belonging to the same $\Ga$-conjugacy class.
\end{lemma}

\begin{proof}
In the case that $\eta_j$ or $\eta_{j+1}$ is not parabolic, one can apply the previous lemma.
So assume that both are parabolic but not $\Ga$-conjugate.
They fix two different cusps $\a$ and $\b$ respectively.
Then $\a$ is the cusp which $F$ and $\eta_j F$ have in common and similarly with $\b$.
For simplicity, we assume $j=1$.
If we vary $\ga z_1$, then the point $z$ as in the last picture, varies on the common boundary of $F$ and $\eta_jF$.
We define $z$ and $w$ as in the picture and by $z', w'$ we denote the corresponding points for $j+1$.
Let a positive constant $C>0$ be given.
We first consider the case that
$$
d(\zeta_{i-1},z)\le C\quad\text{or}\quad d(\zeta_j,z')\le C.
$$
This means that either $z$ or $z'$ are restricted to a compact subset and the same argument as in the proof of Lemma \ref{lem2.3} applies.
We therefore assume $d(\zeta_{i-1},z), d(\zeta_j,z')\ge C$.
Our starting point is the inequality
$$
d(\zeta_{j-1},z)<d(\zeta_{j-1},\ga z_1)
$$
and the same for $\zeta_j$.
We shall show that $d(\zeta_{i-1},z)$ and $ d(\zeta_j,z')$ cannot tend both to infinity, which means that one of them is restricted to a compact set.
If $d(\zeta_{j-1},z)\to\infty$, then also $d(\zeta_{j-1},\ga z_1)\to\infty$ and the same holds for $d(\zeta_j,\ga z_1)$.
This means that $z$ moves into the cuspidal area of $\a$ and so does $\ga z_1$. Assuming $d(\zeta_j,z')\to\infty$ too, we get by the same reasoning that $\ga z_1$ moves into the cuspidal area of $\b$, but as $\a\ne \b$, this is a contradiction.
\end{proof}

\begin{definition}
Let $\ga=\eta_1\cdots \eta_r$ be given in normal presentation. If $\eta_1$ is not parabolic, set $\nu_1=\eta_1$. Otherwise set $\nu_1=\eta_1\cdots\eta_{r_1}$, where $r_1$ is the biggest number such that $\eta_1,\eta_2,\dots,\eta_{r_1}$ all belong to the same parabolic conjugacy class.
Continuing in the same way for $\nu_2,\dots,$ we obtain a new presentation
$$
\ga=\nu_1\cdots\nu_k.
$$
The occurring $\nu_j$ are called the \e{tranches} of $\ga$.
For $g=\smat abcd\in G$ we set
$$
\mu(g)=a^2+b^2+c^2+d^2.
$$
\end{definition}

\begin{lemma}\label{lem2.6}
For every $g\in G$ we have
$$
\frac{\mu(g)}2\le\exp(d(i,g i))\le\mu(g).
$$
\end{lemma}

\begin{proof}
By the Cartan decomposition we can write $g=k_1 Dk_2$ with $k_1,k_2\in \SO(2)$ and $D=\smat a\ \ {a^{-1}}$ for some $a\ge 1$.
Since neither of the expression in the proposition changes under $SO(2)$-multiplication, we can assume that $g=D$.
In this case $d(i,gi)=2\log a$ and we have
$\frac{a^2+a^{-2}}2\le a^2\le a^2+a^{-2}$.
\end{proof}

\begin{proposition}\label{prop2.7}
Let $\ga\in \Ga$ and 
$$
\ga=\eta_1\cdots\eta_r
$$
its normal presentation. The number of tranches $k$ in the presentation can be estimated by
$$
k\le C_\Ga\(\log(\mu(\ga))+1\)
$$
for some constant $C_\Ga$ depending on the group $\Ga$ and the points $z_0,z_1$, but not on the element $\ga$.
\end{proposition}

\begin{proof}
Let $\ga=\nu_1\cdots\nu_k$ be the presentation in tranches.
Using Lemma \ref{lem2.4} we get
$$
d(z_0,\ga z_1)\ge d(\nu_2\nu_1z_0,\ga z_1)+c_2\ge\dots\ge \left[\frac{k}2\right]c_2+d(\ga z_0,\ga z_1)=\left[\frac{k}2\right]c_2+d(z_0,z_1).
$$
The triangle inequality yields
$$
d(z_0,\ga z_1)\le d(z_0,i)+d(i,\ga i)+d(\ga i,\ga z_1)
$$
and thus
$$
k\le C_\Ga\(d(z_0,i)+d(i,\ga i)+d(i,z_1)\).
$$
The claim follows with Lemma \ref{lem2.6}.
\end{proof}

\begin{definition}
For $z,w\in\H$ let
$$
u(z,w)=\frac{|z-w|^2}{\Im z\Im w}.
$$
This number is connected with the hyperbolic distance via the formula
$$
\cosh\(d(z,w)\)=1+\frac12 u(z,w).
$$
\end{definition}

\begin{proposition}\label{prop2.9}
Let $\chi:\Ga\to\GL(V)$ be a finite dimensional representation which is unitary at the cusps.
Then there exists $\sigma_0>1$, such that for all $z,w\in\H$ we have
$$
\norm{\chi(\ga)}=O\(\mu(\ga)^{\sigma_0-1}\)
$$ 
as well as
$$
\norm{\chi(\ga)}=O\((u(z,\ga w)+1)^{\sigma_0-1}\),
$$
where the implied constant depends continuously on $z,w$, but not on $\ga$.
\end{proposition}

\begin{proof}
Let $\ga\in\Ga$ and let $\ga=\nu_1\cdots\nu_k$ be its presentation in tranches.
Then we estimate
$$
\norm{\chi(\ga)}\le\prod_{j=1}^k\norm{\chi(\nu_j)}.
$$
Let $K$ be the maximum of $\norm{\chi(\ga_j)}$ for $j=1,\dots,k$. With Proposition \ref{prop2.7} we obtain 
$$
\norm{\chi(\ga)}\le K^k\le K^{C\log\mu(\ga)+C}=\mu(\ga)^{C\log K}K^C
$$
This implies the first claim.
Similar to Lemma \ref{lem2.6} one finds that there exists $c>0$, which can be chosen independently of $z$ in a compact set, such that
$$
\mu(\ga)\le c\exp(d(z,\ga z)).
$$
By the triangle inequality we get
\begin{align*}
\exp(d(z,\ga z))&\le\exp(d(z,\ga w)+d(\ga w,\ga z))\\
&= \exp(d(w,z))\exp(d(z,\ga w)).
\end{align*}
Since $\frac{\exp}2\le\cosh$ we arrive at $\mu(\ga)=O(u(z,\ga w)+1)$ and whence the claim.
\end{proof}

\begin{lemma}\label{lem2.10}
Assume that $\infty$ is a cusp of $\Ga$ of width one.
Then there exists a constant $C>0$ such that every coset in $\Ga_\infty\bs \Ga$ contains an element $\smat ab mn$ with
$$
a^2+b^2\le C (m^2+n^2).
$$
\end{lemma}

\begin{proof}
Let $\smat\al\beta mn\in\Ga$ be given. We can assume $n>0$ and $m\ne 0$ for otherwise the statement is clear.
Then we find $q\in\Z$ such that $\al=mq+l$, $0\le l<|m|$.
The set $\ga=\smat 1{-q}01$ and $\ga \smat\al\beta mn=\smat abmn$.
Then, as $a=l$ and $an-bm=1$ we get
\begin{align*}
a^2+b^2&=l^2+\(\frac{ln-1}{m}\)^2\\
&\le m^2+\frac{(l+1)^2n^2}{m^2}\le m^2+\frac{(m+1)^2n^2}{m^2}=O(m^2+n^2).\mqed
\end{align*}\end{proof}

\begin{lemma}\label{lem2.11}
Let $\Ga$ have a cusp at $\infty$ of width one. Then with $\sigma_0>1$ as in Proposition \ref{prop2.9}, 
 for $\ga=\smat **cd\in\Ga$ one has $\norm{\chi(\ga)}, \norm{\chi(\ga^{-1})}=O((c^2+d^2)^{\sigma_0-1})$.
The implied constant depends on the representation $\chi$ but not on $\ga$.
\end{lemma}

\begin{proof}
The estimate for $\norm{\chi(\ga)}$ follows from the above lemma and Proposition \ref{prop2.9}.
For the estimate of $\norm{\chi(\ga^{-1})}$ note that if $\ga=\smat abcd$, then $\ga^{-1}=\smat d{-b}{-c}a$, then use Lemma \ref{lem2.10} a second time.
\end{proof}

\begin{lemma}\label{lem2.12}
For $c,d\in\R$ and $z=x+iy\in\C$ we have
$$
\(\frac{y^2}{1+|z|^2}\)(c^2+d^2)\le |cz+d|^2.
$$
In particular with Lemma \ref{lem2.11} we infer, that if $\Ga$ has a cusp at $\infty$ of width 1, then
$$
\norm{\chi(\ga)},\norm{\chi(\ga^{-1})}=O\(\Im(\ga z)^{1-\sigma_0}\(\frac {1+|z|^2}y\)^{\sigma_0-1}\),\quad z\in\H,\ \ga\in\Ga.
$$
\end{lemma}

\begin{proof}
Observe that $|cz+d|^2\ge y^2c^2$ and $|\ol z|^2|cz+d|^2=|c|z|^2+dz|^2\ge d^2y^2$, so that $|cz+d|^2\ge \frac{y^2(c^2+d^2)}{1+|z|^2}.$ This yields
$$ 
\norm{\chi(\ga)},\norm{\chi(\ga^{-1})}=O\(\(\frac{y^2}{1+|z|^2}\)^{1-\sigma_0} |cz+d|^{2(\sigma_0-1)}\),
$$ which implies the claim.
\end{proof}

\section{Eisenstein series}

\begin{definition}
Let $\psi\in C_c^\infty(\R_{>0})$. For a cusp $\a$ we define the \e{incomplete Eisenstein series}
$$
E_\a(z|\psi)=\sum_{\ga\in \Ga_\a\bs\Ga}\psi(\Im(\sigma_\a^{-1}\ga z))\,\chi(\ga^{-1})P_\a.
$$
\end{definition}

For $v\in V$ consider the function $E_\a(\cdot|\psi)v:\H\to V$.
Since $\psi$ has compact support, the function $E_\a(\cdot|\psi)v$ is bounded on the fundamental domain $F$, hence an element of $L^2(F,\chi)=L^2(\Ga\bs\H,\chi)$.

We let $\CE_\a(\Ga\bs\H,\chi)$ denote the closure of all incomplete Eisenstein series in $L^2(\Ga\bs\H,\chi)$, where $\psi$ varies in $C_c^\infty(\R_{>0})$ and $v$ ranges over $V$.
Then we set
$$
\CE(\Ga\bs\H,\chi)=\sum_\a\CE_\a(\Ga\bs\H,\chi).
$$
Note that $\CE_\a(\Ga\bs\H,\chi)=0$ if $\Eig(\chi(\Ga_\a),1)=0$, where for any linear map operator $T$ and $\la\in\C$ we write $\Eig(T,\la)$ for the $\la$-eigenspace of $T$.

\begin{definition}
For $s\in\C$ and a cusp $\a$ we define the \e{Eisenstein series} by
$$
E_\a(z,s,\chi)=\sum_{\ga\in\Ga_\a\bs\Ga}\Im(\sigma_\a^{-1}\ga z)^s\,\chi(\ga^{-1})\, P_\a.
$$
\end{definition}

\begin{proposition}\label{prop3.3}
Let $\sigma_0$ be as in Proposition \ref{prop2.9}.
Then for any cusp $\a$, the Eisenstein series $E_\a(z,s,\chi)$ converges locally uniformly absolutely in the half plane $\Re(s)>\sigma_0$.
\end{proposition}

\begin{proof}
We replace $\a$ with $\infty$ and $\Ga$ with $\sigma_\a^{-1}\Ga\sigma_\a$. Then by Lemma \ref{lem2.12} we have
\begin{align*}
\norm{\Im(\ga z)^s\chi(\ga^{-1})}
&\le \Im(\ga z)^{\Re(s)+1-\sigma_0}\(\frac {1+|z|^2}y\)^{\sigma_0-1}.
\end{align*}
The claim follows.
\end{proof}

In order to estimate the growth of $E_\a(z,s,\chi)$ near a cusp, we need to study its Fourier expansion. We will consider the more general case $E_\a(z|\psi)$ and later set $\psi(t)=t^s$.

\begin{definition}
We write
$$
\ul n(k)=\mat 1k01
$$
and we let $N_\Z$ denote the subgroup of $\PSL_2(\R)$ generated by $\ul n(1)$.
\end{definition}

The integral
$$
K_s(y)\ =\ \frac12\int_0^\infty e^{-y(t+t^{-1})/2}t^s\frac{dt}t
$$
converges locally uniformly absolutely for $y>0$ and $s\in\C$.
The so-defined function $K_s$ is called the \e{$K$-Bessel function}.
It satisfies the estimate
$$
|K_s(y)|\ \le\ e^{-y/2} K_{\Re(s)}(2),\quad\text{ if } y>4.
$$
We also note that the integrand in the Bessel integral is invariant under 
$t\mapsto t^{-1}$, $s\mapsto -s$, so that 
$$
K_{-s}(y)\ =\  K_s(y).
$$

\begin{lemma}
Let $\a,\b$ be cusps of $\Ga$.
We have the disjoint union
$$
\sigma_\a^{-1}\Ga\sigma_\b=\delta_{\a,\b}\Om_\infty\sqcup\bigsqcup_{c>0}\bigsqcup_{d(c)}\Om_{(c,d)},
$$
where $\Om_\infty=N_\Z\om_\infty N_\Z$ and $\Om_{(c,d)}=N_\Z\om_{(c,d)}N_\Z$ with some element $\om_\infty=\smat 1*\ 1\in \sigma_\a\Ga\sigma_\b$ and some $\om_{(c,d)}=\smat **cd\in \sigma_\a^{-1}\Ga\sigma_\b$.

Here $c$ runs over all real numbers $>0$ for which such $\om_{(c,d)}$ exists and $d$ runs modulo $c$.
\end{lemma}

\begin{proof}
This is Theorem 2.7 in \cite{Iwaniec}.
\end{proof}

We now derive the Fourier expansion of the Eisenstein series.
For $\ga\in \sigma_\a^{-1}\Ga\sigma_\b$ we write $\eta(\ga)=\chi(\sigma_\a\ga\sigma_\b^{-1})$.
Writing $\tilde\psi(z)=\psi(\Im(z))$, the double coset decomposition gives
\begin{align*}
E_\a(\sigma_\b z|\psi)
&= \sum_{\ga\in \Ga_\a\bs\Ga}\tilde\psi(\sigma_\a^{-1}\ga\sigma_\b z)\,\chi(\ga^{-1})\,P_\a\\
&= \sum_{\ga\in (\sigma_\a^{-1}\Ga\sigma_a)_\infty\bs \sigma_\a^{-1}\Ga\sigma_\b}\tilde\psi(\ga z)\,\eta(\ga^{-1})\,P_\a\\
&= \delta_{\a,\b}\tilde\psi(y)P_\a+\sum_{c>0}\sum_{d(c)}\(\sum_{k\in\Z} \tilde\psi(\om_{(c,d)}(z+k))\,\chi(\ga_\b)^{-k}\)\eta(\om_{(c,d)}^{-1})\,P_\a \\
\end{align*}
where $\ga_\b$ is a generator of $\Ga_\b$ and so $\chi(\ga_\b)$ is a unitary automorphism of $V$.
So if $1=e(\nu_1),e(\nu_2),\dots,e(\nu_{k(\b)})$ are the eigenvalues where $e(x)=e^{2\pi ix}$ and $\nu_j\in [0,1)$ and $P_1=P_\b,P_2,\dots,P_{k(\b)}$ are the corresponding projections to the eigenspaces, then
$$
\chi(\ga_\b)^{-1}=\sum_{j=1}^{k(\b)} e(-\nu_j)P_j.
$$
Here we use the convention that we list the eigenvalue $1=e(\nu_1)$ even if it doesn't occur, i.e., even in the case when $P_1=P_\b=0$.
We get that $E_\a(\sigma_\b z|\psi)$ equals
$$
\delta_{\a,\b}\tilde\psi(y)P_\a+\sum_{c>0}\sum_{d(c)}
\(\sum_{j=1}^{k(\b)}\sum_{k\in\Z} \tilde\psi(\om_{(c,d)}(z+k))\,e(-k\nu_j)P_j\)\eta(\om_{(c,d)}^{-1})\,P_\a
$$
By the Poisson Summation Formula we get
$$
\sum_{k\in\Z}  \tilde\psi(\om_{(c,d)}(z+k))e(-k\nu_j) P_j
=\sum_{k\in\Z}\int_\R e(t\nu_j)\tilde\psi(\om_{(c,d)}(z+t))e(kt)P_j\,dt.
$$
Writing $\om_{(c,d)}=\smat abcd$ and using $ad-bc=1$ we get $\om_{(c,d)}(z+t)=\frac ac-\frac1{c^2(t+x+iy+d/c)}$.
So that the change of variable $t\mapsto t-x-\frac dc$ yields
$$
\sum_{k\in\Z}e\((k-\nu_j)\(x+\frac dc\)\)\int_\R\psi\(\frac{yc^{-2}}{t^2+y^2}\)e((k+\nu_j)t)\,dt\,P_j.
$$
Hence
$E_\a(\sigma_\b z|\psi)$ equals $\delta_{\a,\b}\psi(y)P_\a$ plus
\begin{multline*}
\sum_{j=1}^{k(\b)}\sum_{k\in\Z}e((k-\nu_j)x)\\
\sum_{c>0}P_j\,\CS_{\a,\b}(k-\nu_j,c,\chi)\int_\R\psi\(\frac{yc^{-2}}{t^2+y^2}\)e((k+\nu_j)t)\,dt\,P_\a,
\end{multline*}
where $\CS_{\a,\b}(r,c,\chi)$ is the \e{Kloosterman sum}
$$
\CS_{\a,\b}(r,c,\chi)
=\sum_{d(c)}e\(r\frac dc\)\eta(\om_{c,d}^{-1}).
$$
We now specialize to the case $\psi(y)=y^s$.
From \cite{Iwaniec} p.205 we take
$$
\int_\R(t^2+y^2)^{-s}\,dt=\pi^{\frac12}\frac{\Ga(s-\frac12)}{\Ga(s)}y^{1-2s}
$$
and for $r\in\R^\times$,
$$
\int_\R(t^2+y^2)^{-s}e(-rt)\,dt=\frac{2\pi^s}{\Ga(s)}\(\frac{|r|}y\)^{s-\frac12}K_{s-\frac12}(2\pi |r|y).
$$
With these notations we have proven the following theorem.

\begin{theorem}\label{thm3.6}
Let $\a,\b$ be cusps of $\Ga$ and let $s$ be in $\C$ with $\Re(s)>\sigma_0$, where $\sigma_0$ is as in Proposition \ref{prop2.9}. Then we have
\begin{align*}
E_\a(\sigma_\b z,s,\chi)&=\delta_{\a,\b} y^sP_\a+\varphi_{\a,\b}(s)y^{1-s}\\
&+ \sum_{k\ne 0}\varphi_{\a,\b,1}(k,s)W_s(kz)\\
&+\sum_{j=2}^{k(\b)}\sum_{k\in\Z}\varphi_{\a,\b,j}(k-\nu_j,s)W_s((k-\nu_j)z),
\end{align*}
where
\begin{align*}
\varphi_{\a,\b}(s)&=\pi^{\frac12}\frac{\Ga(1-\frac12)}{\Ga(s)}P_\a\sum_{c>0}c^{-2s}\CS_{\a,\b}(0,c,\chi)\,P_\a,\\
\varphi_{\a,\b,j}(r,s)&=\frac{\pi^s}{\Ga(s)}|r|^{s-1}P_j\sum_{c>0}c^{-2s}\CS_{\a,\b}(r,c,\chi)\,P_\a,
\end{align*}
and $W_s(z)$ is the \e{Whittaker function}
$$
W_s(z)=2y^{\frac12}K_{s-\frac12}(2\pi y)e(x).
$$
\end{theorem}

\begin{proposition}\label{prop3.7}
For $s$ with $\Re(s)>\sigma_0$, where $\sigma_0$ is as in Proposition \ref{prop2.9}, we have
$$
E_\a(\sigma_\b z,s,\chi)=\delta_{\a,\b} y^sP_\a+\varphi_{\a,\b}(s)y^{1-s}+O(e^{-\beta y})
$$
as $y\to\infty$, where $0<\beta<\min_{j\ge 2}\nu_j$ is arbitrary.
The implied constant depends on $\Ga$, $\chi$ and $s\in\C$ only. It can be chosen to vary continuously in $s$.
\end{proposition}

\begin{proof}
We use the estimate $|W_s(z)|\le C(s)y^\frac12 e^{-\pi y}$ as $y\to\infty$ for some continuous function $C(s)$.
Since for $y\ge\eps>0$ we have
$$
\sum_{k=0}^\infty(k+\nu_j)^{s-1}e^{-\pi y(\nu_j+k)}=O(e^{-\be y}),
$$
the claim follows.
\end{proof}

\begin{proposition}
Let $s\in\C$ with $\Re(s)>\sigma_0$, where $\sigma_0$ is as in Proposition \ref{prop2.9}. With  $\sigma=\Re(s)$ we have
$$
E_\a(\sigma_b z,s,\chi)\ll \frac1{y^\sigma}+y^\sigma,
$$
where the implied constant depends on $\Ga,\chi$ and $\sigma$ only.
\end{proposition}

\begin{proof}
First let $y\ge 1$.
Since 
$$
E_\a(\sigma_\b z,s,\chi)=\delta_{\a,\b}\,y^2\,P_\a+O(y^{1-\sigma}),
$$
we get $E_a(\sigma_\b z,s,\chi)=O\(y^\sigma\)$.
On the other hand, for $z=x+iy$ and $\tilde z=x+i$ the inequality
$$
\Im(z)\Im(\ga z)\le\Im(\ga\tilde z)
$$
holds for $y\le 1$ and arbitrary $\ga\in\PSL_2(\R)$.
This yields the claim.
\end{proof}

\section{Cusp forms}

\begin{definition}
For a cusp $\a$ of $\Ga$ and $f\in L^2(\Ga\bs\H,\chi)$ let
$$
c_0(f,\a,y)=\int_{[0,1]}P_\a f(\sigma_\a(x+iy))\,dx,\qquad y>0,
$$
denote the zeroth Fourier coefficient at the cusp $\a$. By Fubini's theorem, the integral exists almost everywhere in $y$ and defines a measurable function in $y$.

A function $f\in L^2(\Ga\bs \H,\chi)$ is called a \e{cusp form}, if 
$$
c_0(f,\a,y)=0
$$
holds almost everywhere in $y\in (0,\infty)$.
The space of cusp forms will be denoted by $L^2_\cusp(\Ga\bs \H,\chi)$.
\end{definition}

\begin{definition}
Let $k\in C^\infty(\R_{> 0})$.
For $z,w\in\H$ recall the notation $u(z,w)=\frac{|z-w|^2}{\Im(z)\Im(w)}$. Then $\cosh\(d(z,w)\)=1+\frac12 u(z,w)$.
By abuse of notation we write $k(z,w)=k(u(z,w))$.
We define an integral operator $L=L_k$ by
$$
Lf(z)=\int_\H k(z,w)f(w)\,dw.
$$
For $f\in L^2(\Ga\bs \H,\chi)$ we obtain
$$
Lf(z)=\int_{\Ga\bs\H}K(z,w)f(w)\,dw,
$$
where $K(z,w)=\sum_{\ga\in\Ga}k(z,\ga w)\chi(\ga)$.
Note that for $\ga,\tau\in\Ga$ we have
$$
K(\ga z,\tau w)=\chi(\ga)K(z,w)\chi(\tau^{-1}).
$$
\end{definition}

\begin{definition}
Let $C(\Ga\bs\H,\chi)$ denote the space of all continuous functions $f:\H\to V$ with $f(\ga z)=\chi(\ga)f(z)$ for every $z\in\H$, $\ga\in\Ga$.
Write $C_\cusp(\Ga\bs\H,\chi)$ for the subspace of all $f$ such that 
$$
\int_{[0,1]}P_\a f(\sigma_\a (x+iy))\,dx=0
$$
for every cusp $\a$ and every $y>0$.
We call $C_\cusp(\Ga\bs\H,\chi)$ the space of \e{continuous cusp forms}.
\end{definition}

\begin{proposition}
Suppose the kernel satisfies
$$
k(z,w)=\ll u(z,w)^{-\sigma}
$$
for $\sigma>\sigma_0$ and $\sigma_0$ as in Proposition \ref{prop2.9}.
Then $K(z,w)$ is continuous on the domain $\Big\{ (z,w)\in\H\times\H: z\nequiv w\mod\Ga\Big\}$.
If more strongly, $k$ satisfies
$$
k(z,w)=\ll(u(z,w)+1)^{-\sigma}
$$
then $K$ is continuous on all of $\H\times \H$ and the integral $Lf$ converges locally uniformly to an element of $C(\Ga\bs \H,\chi)$, defining a linear operator $L:L^2(\Ga\bs\H,\chi)\to C(\Ga\bs\H,\chi)$.
This operator maps the space of cusp forms $L^2_\cusp(\Ga\bs \H,\chi)$ to the space $C_\cusp(\Ga\bs\H,\chi)$ of continuous cusp forms.
\end{proposition}

\begin{proof}
Replacing $\Ga$ with a conjugate, we may assume that $\infty$ is a cusp of width one.
By Proposition \ref{prop2.9} we have
$$
\norm{\chi(\ga)}=O((u(\ga z,w)+1)^{\sigma_0-1}).
$$ 
We then get for $z\nequiv w\mod\Ga$ that
\begin{align*}
\sum_{\ga\in\Ga}\norm{k(z,\ga w)\chi(\ga)}&\ll\sum_{\ga\in\Ga}u(a,\ga w)^{-\sigma}\norm{\chi(\ga)}\\
&\ll\sum_{\ga\in\Ga} u(z,\ga w)^{\sigma_0-1-\sigma},
\end{align*}
where the implied constants depend continuously of $(z,w)$.
By Lemma 2.11 in \cite{Iwaniec}, locally uniform convergence follows if $\sigma>\sigma_0$.
The convergence of the integral $Lf$ follows from this bound.

For the last assertion let $f$ be a cusp form, $g=Lf$ and $\ul n(t)=\smat 1t\ 1$.
We compute
\begin{align*}
\int_{[0,1]}P_\a g(\sigma_\a\ul n(t)z)\,dt
&=\int_{[0,1]}\int_\H k(\sigma_\a\ul n(t)z,w)P_\a f(w)\,dw\,dt\\
&=\int_\H k(z,w)\(\int_{[0,1]}P_\a f(\sigma_\a\ul n(t)w)\,dt\)\,dz=0.\mqed
\end{align*}
\end{proof}

\begin{definition}
For a cusp $\a$, the zeroth Fourier coefficient of $K(z,w)$ which equals
$$
H_\a(z,w)=\sum_{\ga\in \Ga_\a\bs \Ga}\int_\R k(z,\sigma_\a\ul n(t)\sigma_\a^{-1}\ga w)\,dt P_\a\chi(\ga),
$$
is also called the \e{principal part} of the kernel $K$ at the cusp $\a$.
\end{definition}

\begin{proposition}
The operator with kernel $H_\a$ annihilates cusp forms, i.e., for every cusp form $f$ we have
$$
\int_{\Ga\bs\H}H_\a(z,w)f(w)\,dw=0.
$$
\end{proposition}

\begin{proof}
By unfolding the integral we get
\begin{align*}
&\int_{\Ga\bs\H}H_\a(z,w)f(w)\,dw\\
&=\int_{\Ga\bs\H}\sum_{\ga\in \Ga_\a\bs \Ga}\int_\R k(z,\sigma_\a\ul n(t)\sigma_\a^{-1}\ga w)\,dt P_\a\chi(\ga)f(w)\,dw\\
&=\int_{\Ga_\a\bs\H}\int_\R k(z,\sigma_\a\ul n(t)\sigma_\a^{-1} w)\,dt P_\a f(w)\,dw\\
&=\int_{\Ga_\a\bs\H}\int_{[0,1]}\sum_{\ga\in\Ga_\a} k(z,\sigma_\a\ul n(t)\sigma_\a^{-1}\ga w)\,dt P_\a f(w)\,dw\\
&=\int_{\H}\int_{[0,1]} k(z,\sigma_\a\ul n(t)\sigma_\a^{-1}w)\,dt P_\a f(w)\,dw\\
&=\int_{\H} k(z,w)  \(\int_{[0,1]}P_\a f(\sigma_\a\ul n(-t)\sigma_\a^{-1}w)\,dt\)\,dw=0.\mqed
\end{align*}
\end{proof}

\begin{definition}
We add the principal parts
$$
H(z,w)=\sum_\a H_\a(z,w).
$$
The kernel function
$$
\what K(z,w)=K(z,w)-H(z,w)
$$
is called the \e{compact part} of $K(z,w)$.
If $L$ is the operator with kernel $K$, then $\what L$ will denote the operator with kernel $\what K$.
On the space of cusp forms, the operator $\what L$ coincides with $L$.
\end{definition}

\begin{theorem}\label{thm4.6}
Suppose that
$$
k(z,w)\ll (u(z,w)+1)^{-\sigma}
$$
for some $\sigma>\sigma_0$ and $\sigma_0$ as in Proposition \ref{prop2.9}.
Then the kernel $\what K$ is square integrable on $F\times F$ and thus defines a Hilbert-Schmidt operator on $L^2(F,\chi)$.
\end{theorem}

The Proof of the theorem will take the rest of the section.

\begin{lemma}
Suppose the kernel satisfies
$$
k(z,w)\ll(u(z,w)+1)^{-\sigma}
$$
for $\sigma>\sigma_0$ and $\sigma_0$ as in Proposition \ref{prop2.9}.
\begin{enumerate}[\rm (a)]
\item For $z,w\in F$ we have
$$
K(z,w)=k(z,w)+\sum_{\a\in\CC(F)} \sum_{\substack{\ga\in \Ga_\a\\ \ga\ne 1}}k(z,\ga w)\chi(\ga)+O(1),
$$
where the first sum runs over the finite set $\CC(F)$ of cusps of the fundamental domain $F$.
\item For any cusp $\a$ of $F$ we have
$$
H_\a(z,w) =\int_\R k(z,\sigma_\a\ul n(t)\sigma_\a^{-1} w)\,dt\,P_\a +H'_\a(z,w),
$$
where $H_\a'$ has bounded $L^2$-norm on $F\times F$.
\end{enumerate}
\end{lemma}

\begin{proof}
(a) We have
$$
K(z,w)=\sum_{\a\in\CC(F)}\sum_{1\ne \ga\in\Ga_\a}k(z,\ga w)\chi(\ga)+\sum_{\substack{\ga\in\Ga\\ \text{not parabolic}}}k(z,\ga w)\chi(\ga).
$$
For a cusp $\a$ we write $\chi_\a:\sigma_\a^{-1}\Ga\sigma_\a\to\GL(V)$, $\chi_\a(\ga)=\chi(\sigma_\a\ga\sigma_\a^{-1})$.
We write $N(\Z)=\smat 1\Z\ 1$ and get
$$
\sum_{\substack{\ga\in\Ga\\ \text{not parabolic}}}k(z,\ga w)\chi(\ga)=
\sum_{\substack{\ga\in N(\Z)\bs\sigma_\a^{-1}\Ga\sigma_1\\ \text{not parabolic}}}
\sum_{j\in\Z}k(z,\ga w+j)\chi_\a(\ul n(j)\ga).
$$
Now
\begin{align*}
k(z,\ga w+j)&\ll (1+u(z,\ga w+j))^{-\sigma}\\
&=\(1+\frac{(\Re(z)-\Re(\ga w)-j)^2+(\Im(z)-\Im(\ga w))^2}{\Im(z)\Im(\ga w)}\)^{-\sigma}\\
&\ll \frac{\Im(\ga w)^\sigma}{\(\Im(z)+\frac{(\Re(z)-\Re(\ga w)-j)^2}{\Im(z)}\)^\sigma},
\end{align*}
which yields
\begin{align*}
\sum_{j\in\Z}k(z,\ga w+j)&\ll \Im(\ga w)^\sigma\sum_{j\in\Z}\(\Im(z)+\frac{(\Re(z)-\Re(\ga w)-j)^2}{\Im(z)}\)^{-\sigma}\\
&\ll \Im(\ga w)^\sigma\sum_{j=0}^\infty \(\Im(z)+\frac{j^2}{\Im(z)}\)^{-\sigma}.
\end{align*}
If we assume $\Im(z)\ge A$ for some $A>0$, then we get
\begin{align*}
&\sum_{j=0}^\infty \(\Im(z)+\frac{j^2}{\Im(z)}\)^{-\sigma}\\
&<\Im(z)^{-\sigma}+\int_0^\infty \(\Im(z)+\frac{u^2}{\Im(z)}\)^{-\sigma}\,du\ll \Im(z)^{-\sigma+1},
\end{align*}
and thus
$$
\sum_{j\in\Z}k(z,\ga w+j)\ll \Im(z)^{1-\sigma}\Im(\ga w)^\sigma.
$$
We therefore get
$$
\norm{\sum_{\substack{\ga\in\Ga\\ \text{not parabolic}}}k(\sigma_\a z,\ga\sigma_\a w)\chi(\ga)}
\ll\Im(z)^{1-\sigma}\sum_{\substack{\ga\in\Ga_\a\bs\Ga\\ \text{not parabolic}}}\Im(\sigma_\a^{-1}\ga w)^\sigma\norm{\chi(\ga)}.
$$
The latter sum is $O(y^{1-\sigma})$ by Proposition \ref{prop3.7}.
So we see that the sum $\sum_{\substack{\ga\in\Ga\\ \text{not parabolic}}}k(\sigma_\a z,\ga\sigma_\a w)\chi(\chi)$ is uniformly bounded for $z$ and $w$ in $F$ with $\Im(\sigma_\a^{-1} z)\ge A$.
Since this is true for any cusp, the claim follows.

Part (b) is obtained in a similar fashion.
\end{proof}

In order to finish the proof of Theorem \ref{thm4.6} it now remains to show that
$$
J_\a(z,w)=\sum_{\ga\in\Ga_\a} k(z,\ga w)\chi(\ga)-\int_\R k(z,\sigma_\a\ul n(t)\sigma_\a^{-1}w)\,dt\, P_\a
$$
is bounded on $F\times F$.
We first consider
$$
J_\a(z,w)P_\a=\sum_{\ga\in\Ga_\a} k(z,\ga w)P_\a-\int_\R k(z,\sigma_\a\ul n(t)\sigma_\a^{-1}w)\,dt\, P_\a.
$$
For $\psi(t)=t-[t]-\frac12$ and $f$ continuously differentiable on $\R$ with $f(t),tf'(t)\in L^1(\R)$ integration by parts shows that
$$
\sum_{j\in\Z}f(k)=\int_\R f(t)\,dt+\int_\R\psi(t)f'(t)\,dt.
$$ 
In our case this yields
\begin{align*}
|J_\a(\sigma_\a z,\sigma_\a w)P_\a|&=\left|\sum_{j\in\Z}k(z,w+j)P_\a-\int_\R k(z,\ul n(t)w)\,dt\,P_\a\right|\\
&=\left|\int_\R\psi(t)\frac\partial{\partial t}k(u(z,w+t))\,dt\,P_a\right|\\
&=\left|\int_\R\psi(t)\frac\partial{\partial t}k\(\frac{|w+t-z|^2}{\Im(z)\Im(w)}\)\,dt\,P_a\right|\ll \int_\R |k'(u)|\,du\ll 1.
\end{align*}
Next we consider the contribution of $J_\a(z,w)$ on the orthogonal complement $\Eig(\chi(\ga_\a),1)^\perp$, where $\ga_\a$ is a generator of $\Ga_\a$.
In this case we have
\begin{align*}
(\Id-\chi(\ga_\a))\sum_{\ga\in\Ga_\a}k(\sigma_\a z,\ga\sigma_a w)\chi(\ga)
&= \sum_{j\in\Z}\(k(w,z+j)-k(w,z+j-1)\)\chi(\ga_\a)^n\\
&\ll \int_\R dk(w,z+t)\ll \int_\R|k'(u)|\,du\ll 1.
\end{align*}
It follows that $\what K$ is an $L^2$-kernel and thus defines a Hilbert-Schmidt operator.
Theorem \ref{thm4.6} is proven.
\hfill $\square$

For the next theorem let's shortly recall spectral theory of compact operators.
For a compact operator $T:H\to H$ every spectral value $\la\ne 0$ is an isolated point in the spectrum and it is a generalized eigenvalue where the generalized eigenspace
$$
E(\la)=\bigcup_{n=1}^\infty\ker(T-\la)^n
$$
is non-zero and finite-dimensional.
For $\la\ne 0$ in the spectrum of $T$ let 
$$
P(\la)=\frac1{2\pi i}\int_\eta(T-z)^{-1}dz
$$
where $\eta$ is a closed path in $\C$ which surrounds the eigenvalue $\la$ once and no other spectral value.
Then $P$ is a continuous projection with image $E(\la)$, called the \e{Riesz-projection} of the eigenvalue $\la$.
Let $N$ be the intersection of all $P(\la)$, $\la\in\sigma(T)$, $\la\ne 0$.
Then $N$ ist preserved by $T$ which is quasi-nilpotent on $N$.
The space $N$ is called the \e{nilpotence kernel} of $T$.
The direct sum
$$
N\oplus\bigoplus_{\la\ne 0}E(\la)
$$
is dense in $H$.
If $S$ is an operator commuting with $T$, it will preserve each of the spaces $E(\la)$ and $N$.

\begin{theorem}
[Spectral decomposition]
Let $\Delta_\cusp$ denote the restriction of the hyperbolic Laplacian $\Delta$ to the space of cusp forms.
For every $\la\in\C$ there exists $m\in\N_0$  such that
$$
\ker(\Delta_\cusp-\la)^m=\ker(\Delta_\cusp-\la)^{m+k}
$$
holds for every $k\in\N$.
Let $m(\la)$ denote the smallest such $m\in\N_0$ and let $H_\cusp(\la)$ denote the space $\ker(\Delta_\cusp-\la)^{m(\la)}$.
Then there is a sequence $\la_j\in\C$, tending to infinity, such that $H_\cusp(\la)=0$ unless $\la=\la_j$ for some $j$ and the direct sum
$$
N\oplus \bigoplus_{j=1}^\infty H_\cusp(\la_j)
$$
is dense in $L^2_{\cusp}(\Ga\bs\H,\chi)$, where $N$ is the intersection of all nilpotence kernels of the operators $R(f)$.
On the closed space $N$, every $R(f)$ has only the spectral value $0$.
For each even $f\in C_c^\infty(\R)$ the operator with kernel $k(z,w)=f(u(z,w))$ is a trace class operator on $L^2_\cusp(\Ga\bs\H,\chi)$.
\end{theorem}

We believe that the space $N$ is zero, but we can't prove this statement.

\begin{proof}
Let $\C_c^\infty(\R)^\even$ be the set of all smooth functions $f$ of compact support which are even, i.e., satisfy $f(-x)=f(x)$.
For such $f$ the kernel $k(z,w)=f(u(z,w))$ is smooth satisfies the requirement of Theorem \ref{thm4.6}, thus defines a Hilbert-Schmidt operator on $L^2_\cusp(\Ga\bs\H,\chi)$.
The algebra of all such operators is commutative (Section 11.2 in \cite{HA2}).
So it has a common spectral decomposition as a direct sum of a nilpotence kernel and generalized eigenspaces, which are all finite-dimensional. These generalized eigenspaces are also generalized eigenspaces of the Laplace operator \cite{Iwaniec}.
As for the last statement about the trace class, note that so far we know that these operators are Hilbert-Schmidt.
The algebra of these operators coincides with the convolution algebra of $K$-bi-invariant functions in $C_c^\infty(G)$.
By the Theorem of Dixmier-Malliavin \cite{DixMal} we have that every $f\in C_c^\infty(G)$ can be written as a finite sum $f=\sum_{j=1}^n g_j*h_j$, where $g_j,h_j\in C_c^\infty(G)$.
If $f$ is $K$-bi-invariant, we can integrate over $K$ and assume the each $g_j$ and each $h_j$ is $K$-bi-invariant, too.
This means that the operator induced by $f$ is a finite sum of products of Hilbert-Schmidt operators, hence trace class.
\end{proof}

\section{Meromorphic continuation of the Eisenstein series}

Let $G_s(u)$ be the integral
$$
G_s(u)=\frac1{4\pi}\int_0^1(\xi(1-\xi))^{s-1}(\xi+u)^{-s}\,d\xi.
$$

\begin{proposition}\label{prop5.1}
The integral $G_s(u)$ converges absolutely for $\Re(s)>0$.
The function $G_s(u)$ is smooth and satisfies the differential equation
$$
\Delta+s(1-s))G_s(u)=0,
$$
where $\Delta$ is applied to either $z$ or $w$ in $u=u(z,w)$.
We have the following bounds
\begin{align*}
G_s(u)&=-\frac1{4\pi}\log(u)+O(1),&u\to 0,\\
G_s'(u)&=-\frac1{4\pi u}+O(1),&u\to 0,\\
G_s(u)&\ll u^{-\Re(s)}&u\to\infty.
\end{align*}
\end{proposition}

\begin{proof}
This is Lemma 1.7 in \cite{Iwaniec}.
\end{proof}

\begin{theorem}\label{thm5.2}
For $\Re(s)>0$ let $R_s$ be the integral operator
$$
R_sf(z)=-\int_\H G_s(u(z,w))f(w)\,dw.
$$
If $f:\H\to V$ is smooth and together with its first and second order derivatives (under the Lie algebra) satisfies $\norm f\ll y^{\sigma}+y^{-\sigma}$ for some $\sigma>0$ and if $\Re(s)>1+\sigma$, then
$$
(\Delta+s(1-s))R_sf=f=R_s(\Delta+s(1-s))f.
$$
\end{theorem}

\begin{proof}
In \cite{Iwaniec}, Theorem 1.17. and Lemma 1.18 we find the proof if $f$ is bounded with bounded derivatives.
For the general situation note that the growth condition assures that the integral $R_sf$ converges absolutely and defines a smooth function. Let $1=\sum_{j=1}^\infty u_j$ be a smooth partition of unity in $\H$ with compact supports.
Then
\begin{align*}
(\Delta+(s(1-s))R_sf&=\sum_{j=1}^\infty (\Delta+(s(1-s))R_s(u_jf)=\sum_{j=1}^\infty u_jf=f
\end{align*}
and analogously for the other order.
\end{proof}

Note that the singularity of $G_s(u)$ at $u=0$ does not depend on $s$. It will therefore vanish in differences of the form $G_s-G_a$ for $a\ne s$.

\begin{theorem}
Let $\a$ be a cusp of $\Ga$.
The Eisenstein series $E_\a(z,s,\chi)$ extends to a function $\H\times\C\to \End(V)$, which is smooth in $z$ and  meromorphic in the complex variable $s$.
Let $\varphi_{\a,\b}(s)$ the function in the constant term of the Fourier expansion in Theorem \ref{thm3.6} and let 
$$
\Phi(s)=\(\varphi_{\a,\b}(s)\)_{\a,\b\in\CC(\Ga)},
$$
where $\CC(\Ga)$ is the set of inequivalent cusps of $\Ga$.
Then $\det\Phi(s)\ne 0$, the function $\Phi(s)$ extends to a meromorphic function on $\C$ and the vector $\CE(z,s,\chi)=\(E_\a(z,s,\chi)\)_\a$ satisfies the functional equation
$$
\CE(z,s,\chi)=\Phi(s)\CE(z,1-s,\chi).
$$
\end{theorem}

\begin{proof}
It turns out, that the proof of the untwisted case, i.e., with $\chi=1$, as in Chapter 6 of \cite{Iwaniec}, essentially transfers to the twisted case.
However, the argument has to be changed significantly in one point.
Iwaniec uses Fredholm's theory of integral operators, in particular a power series expansion of the resolvent (Appendix A.4 in \cite{Iwaniec}). The latter cannot be used for Integral operators in vector bundles, as it is based on a point-wise estimate of integral kernels using Hadamard's inequality (see formula (A.14) in \cite{Iwaniec}).
This formula has to be replaced with the following argument from the theory of pseudo differential operators.

Let $\la\in\C$ and $F\subset \H$ the given fundamental domain.
We consider a given kernel $K:F\times F\to \End(V)$.
By abuse of notation we write $K$ for the induced operator as well.
Assuming the operator norm $\norm{K(z,w)}_\op$ to be bounded, we conclude that the operator $K$ is Hilbert-Schmidt, as $F$ has finite volume. Therefore, $K$ is compact and the map $\C\to\CB(L^2(\Ga\bs\H,\chi))$, $\la\mapsto (I-\la K)^{-1}$ is holomorphic 
outside a countable set of poles which can only accumulate at $\infty$.
Let $T_\la(z,w)$ denote the kernel of $K(I-\la K)^{-1}$. If the kernel $K$ is smooth, then $K$ is a smoothing operator, i.e., a pseudo differential operator of symbol class $S^\infty$ \cite{Shu} and as the smoothing operators form an ideal, the operator $T_{\la}$ is a smoothing operator whenever defined, i.e., as long as $\frac1\la$ is not an eigenvalue of $K$. Therefore, the kernel $T_\la(z,w)$ is holomorphic in $\la$ and smooth in $z,w$.
Now, as in \cite{Iwaniec}, for given $f$ one considers solutions of the Fredholm equation 
$$
g=f+\la Kg,
$$
i.e., $g=(I-\la K)^{-1}f$, or
$$
g=f+\la K(I-\la K)^{-1} f.
$$
The point about this last equation is, that  the right hand side is a meromorphic function in $\la$, defined on the entire complex plane except for countably many poles. As in \cite{Iwaniec}, it is found that after suitable modifications, the Eisenstein series $E_\a(z,s,\chi)$ takes the place of $g$ and the analytic continuation in $\la$, which is a polynomial in $s$, yields the analytic continuation of the Eisenstein series.

For the functional equation we argue as follows. Let $F(s,z)$ be the difference
$$
F(s,z)=\CE(z,s,\chi)-\Phi(s)\CE(z,1-s,\chi).
$$
The Fourier expansion shows that $F(s,z)$ is rapidly decreasing at every cusp, hence it is in $L^2(\Ga\bs\H,\chi)$.
Further we have $\Delta F(s,z)=s(1-s)F(s,z)$.
Let $G_{a,b}=G_a-G_b$, then it follows that
$$
G_{a,b}F(s,z)=\(\frac1{s(1-s)-a(1-a)}-\frac1{s(1-s)-b(1-b)}\)F(s,z).
$$
The singularity on the diagonal of $G_a$ vanishes in the difference $G_a-G_b$  by Proposition \ref{prop5.1}.
If $\Re(a)$ and $\Re(b)$ are sufficiently large, by Proposition \ref{prop2.9} and Proposition \ref{prop5.1}, the sum $\sum_{\ga\in\Ga}G_{a,b}(z,\ga w)$ converges locally uniformly and defines a continuous integral operator on $L^2(\Ga\bs\H,\chi)$.
Now assume that $F(s,z)$ is not zero. Then, being analytic in $s$, it can only vanish for $s$ in a discrete set. Therefore, by the above equation, the operator $G_{a,b}$ has an unbounded set of eigenvalues which contradicts its boundedness.
It follows that $F(s,z)$ must be zero.
\end{proof}

{\small Mathematisches Institut\\
Auf der Morgenstelle 10\\
72076 T\"ubingen\\
Germany\\
\tt deitmar@uni-tuebingen.de\\
frank.monheim@mailbox.org}

\today

\end{document}